\documentclass[11pt]{amsart}

\usepackage{amsmath,amssymb,amsthm,mathtools}
\usepackage{enumitem}
\usepackage[colorlinks=true,linkcolor=blue,citecolor=blue,urlcolor=blue]{hyperref}

\allowdisplaybreaks
\numberwithin{equation}{section}

\newcommand{\ddc}{\sqrt{-1}\partial\bar\partial}
\newcommand{\Ric}{\operatorname{Ric}}
\newcommand{\BK}{\operatorname{BK}}

\newcommand{\vep}{\varepsilon}

\newcommand{\e}{\mathrm e}

\renewcommand{\d}{\partial}

\newcommand{\ddbar}{\sqrt{-1}\d\overline{\d}}

\theoremstyle{plain}
\newtheorem{theorem}{Theorem}[section]
\newtheorem{proposition}[theorem]{Proposition}
\newtheorem{lemma}[theorem]{Lemma}

\theoremstyle{remark}

\title[The top Yau--Yang conjecture]{The top Yau--Yang conjecture for K\"ahler manifolds with positive sectional curvature}

\author[V. Datar]{Ved V. Datar}
\address{Department of Mathematics, Indian Institute of Science, Bengaluru 560012, India}
\email{vvdatar@iisc.ac.in}

\author[V. P. Pingali]{Vamsi Pritham Pingali}
\address{Department of Mathematics, Indian Institute of Science, Bengaluru 560012, India}
\email{vamsipingali@iisc.ac.in}

\author[H. Seshadri]{Harish Seshadri}
\address{Department of Mathematics, Indian Institute of Science, Bengaluru 560012, India}
\email{harish@iisc.ac.in}

\date{\today}

\subjclass[2020]{Primary 53C55; Secondary 32Q15, 32U40, 32W20}
\keywords{Cohn--Vossen inequality, positive sectional curvature, Ricci form, Bedford-Taylor theory}

\begin{document}

\begin{abstract}
We prove that the top wedge power of the Ricci form of a complete non-compact K\"ahler manifold with positive sectional curvature has finite integral. Using a result of Chen-Zhu, an immediate consequence is the quasiprojectivity of such manifolds under the assumption of bounded sectional curvature. A key new idea to prove B\'ezout estimates along with a Lipschitz weight with finite Monge-Amp\`ere mass is used in the proof of the main result.
\end{abstract}

\maketitle

\section{Introduction}

If $(M,g)$ is a complete Riemannian $2$-manifold with Gaussian curvature $K > 0$, the classical Cohn--Vossen theorem \cite{CV} implies that $\int_M K dV < \infty$. In higher dimensions, a Cohn--Vossen type problem, posed by Yau \cite{Yau} and subsequently formulated in the K\"ahler setting by Yang \cite{Yang}, asks if the following statement is true: Let $(X^n,\omega)$ be a complete K\"ahler manifold with nonnegative holomorphic bisectional curvature. If $\Ric_{\omega}$ denotes the Ricci form and $r$ denotes the distance from a fixed point $o \in X$, then the quantities
\[
        r^{2k-2n}\int_{B(o,r)}\Ric_{\omega}^k\wedge\omega^{n-k},
        \qquad 1\le k\le n,
\]
are bounded independently of $r$. When the sectional curvatures are nonnegative, the case $k=1$ follows from a more general result of Petrunin \cite{Pet}. The case $k=n$, equivalent to global finiteness of the top Ricci mass
\[
        \int_X \Ric_{\omega}^n,
\]
is of particular interest, as it plays a role in Yau's uniformisation conjecture \cite{CZ2, DPS, DPS2} and in compactification results for non-compact K\"ahler manifolds \cite{Mok90}.   Liu \cite{GL} proved the finiteness of $\int_X \Ric_{\omega}^n$ when the volume growth is maximal.

In \cite{DPS}, we proved the surface case of this finiteness statement under positive sectional curvature (without any volume growth hypotheses). The key result on which our proof depended is actually valid in all dimensions: if $(X^n,\omega)$, $n \ge 1$, has positive holomorphic bisectional curvature and $X$ admits a smooth strictly plurisubharmonic exhaustion function with bounded gradient, then there exists a uniformly Lipschitz plurisubharmonic function $\varphi$ such that
\begin{gather}\label{ineq:finiteMAmassweight}
    0<\displaystyle \int_X (\ddc\varphi)^n<\infty .
\end{gather}

The hypotheses needed for this construction, together with the heat-flow and cut-off inputs are satisfied in the positive sectional curvature case by the results used in \cite{DPS,DPS2} and by Greene--Wu \cite{GW}. The purpose of this work is to show that the top-degree Yau--Yang estimate holds in every dimension:

\begin{theorem}\label{thm:main}
Let $(X^n,\omega)$ be a complete non-compact K\"ahler manifold with positive sectional curvature. Then
\[
        \int_X \Ric_{\omega}^n<\infty .
\]
More generally, the same conclusion holds if $(X, \omega)$ has positive holomorphic bisectional curvature and $X$ admits a smooth strictly plurisubharmonic exhaustion function with uniformly bounded gradient.
\end{theorem}

By a theorem of Chen--Zhu \cite{CZ2}, a complete non-compact K\"ahler manifold with positive and bounded sectional curvature and finite top Ricci mass is biholomorphic to a quasiprojective variety. Thus Theorem \ref{thm:main} immediately gives the following.

\begin{theorem}\label{thm:ChenZhucorollary}
Let $(X^n,\omega)$  be a complete non-compact K\"ahler manifold with positive and bounded sectional curvature. Then $X$ is biholomorphic to a quasiprojective variety.
\end{theorem}

In fact, in \cite{DPS2} we established Theorem \ref{thm:ChenZhucorollary} without the boundedness assumption on sectional curvature. Note that in this dimension, it follows by a result of Ramanujam \cite{R} that $X$ is actually biholomorphic to $\mathbb{C}^2$.  In work in progress, we shall use the techniques in this paper to remove the boundedness assumption in Theorem \ref{thm:ChenZhucorollary}. As remarked in \cite{DPS2}, a Ramanujam-type result for higher $n$ under curvature hypotheses is a challenging open problem.

\indent The proof of Theorem \ref{thm:main} takes inspiration (just as in \cite{DPS}) from \cite{CZ3}. Just as in \cite{DPS}, the weight $\varphi$ satisfying \eqref{ineq:finiteMAmassweight} is related to the Ricci form using a Poincar\'e-Lelong type formula. Such a relationship is delicate because of the non-smoothness of the weight and we use Bedford-Taylor theory in an essential way. The main innovation arises in the B\'ezout-type estimates for Poincar\'e-Lelong integrals. In \cite{DPS}, a delicate integration-by-parts argument involving a Bochner identity was used. We bypass that argument in this paper whilst noting that it may still be useful as an independent argument. Instead, as we explain in Section \ref{sec:proof}, we introduce a linear growth function $H$, and an induction argument to prove the desired estimates. This argument was produced after discussion with ChatGPT 5.5 Pro. To our knowledge, this is perhaps the first instance of AI coming up with a truly original idea in geometric analysis. After probing further, we realised that it may have been motivated by the ODE-type methods used in Ohsawa-Takegoshi-type theorems. That being said, our experiments (by asking it to prove the results of \cite{DPS} without using literature after 2024) showed that it is still unable to come up with the road-map of the proof of \cite{DPS} on its own.\\

\emph{Acknowledgements}: The second-named author Pingali is grateful to ANRF for support via a MATRICS grant (ANRF/ARGM/2025/000313/MTR). The authors also acknowledge partial support by the DST FIST program-2021
[TPN-700661]. We also thank ChatGPT 5.5 Pro for its indispensable contributions to this paper.

\section{Preliminaries and background}\label{sec:prelim}

In this section we summarise some important results used in the proof of Theorem \ref{thm:main}.

\indent In \cite{DPS} we proved the following important result by solving a Monge-Amp\`ere equation.

\begin{proposition}[Datar--Pingali--Seshadri]\label{prop:weight}
Let $(X^n,\omega)$ be either a complete non-compact K\"ahler manifold with positive sectional curvature, or a complete non-compact K\"ahler manifold with $\BK_\omega>0$ admitting a smooth strictly plurisubharmonic exhaustion function with uniformly bounded gradient. Then there exists a uniformly Lipschitz strictly plurisubharmonic function $\varphi$ on $X$ such that, with
\[
        \alpha:=\ddc\varphi,
\]
one has
\[
        0< A_0:=\int_X \alpha^n<\infty .
\]
\end{proposition}
\vspace{5mm}

This weight was used in \cite{DPS,DPS2} to produce weighted-$L^2$ holomorphic sections of $K_X$. The non-smoothness of this weight is problematic and was addressed in \cite{DPS, DPS2} by heat-flow regularization which we summarize below.

\indent Under either curvature hypothesis in Theorem \ref{thm:main}, $\Ric_{\omega} \geq 0$; hence there exists a unique, positive, symmetric and stochastically complete heat kernel $H(x,y,t)$. If $\varphi$ is a Lipschitz function on $X$, we let $$u(x,t) = \int_X H(x,y,t)\varphi(y)\,dy.$$

 Then $u(x,t)$ is a solution to the heat equation with initial condition $u(x,0)=\varphi(x)$. The heat-flow regularization used in \cite{DPS,NT,NT2} gives that $u_t(\cdot):=u(\cdot,t)$ is strictly psh for each $t>0$, and there exists a constant $A>0$ such that
 \begin{align*}
 |\nabla u_t|,~t|\ddc u_t| \leq A.
  \end{align*}

Note however that the smooth weights $u_t$ may no longer have finite Monge-Amp\`ere mass.  To switch back and forth between the smooth and non-smooth weights, we need the following crucial estimate \cite{DPS} (which was incidentally proven by ChatGPT 5.0 in the thinking mode, and was inspired from the techniques in \cite{Grig}).

\begin{lemma}\label{lem:heat-est}
There exists a dimensional constant $c(n)$ such that for any $0\leq t_1\leq t_2$, $$u_{t_2}(x) \leq u_{t_1}(x) + Ac\sqrt{t_2-t_1}$$ for all $x \in X$, where $A$ is the Lipschitz constant for $\varphi$.
\end{lemma}
\vspace{5mm}

To compare the Ricci form with $\partial \bar{\partial} \varphi$, we need $L^2$-holomorphic sections. To this end we need the following standard result \cite{DPS}.

\begin{proposition}[$L^2$ canonical section]\label{prop:section}
Let $u_t$, $0\le t\le1$, be the heat-flow smoothing of $\varphi$, with $u_0=\varphi$. There exist $q\gg1$, a non-zero holomorphic section
\[
        s\in H^0(X,K_X),
\]
and a constant $C$ such that
\[
        \int_X \|s\|^2_{q u_t}\,\omega^n\le C,
        \qquad 0\le t\le1.
\]
Here
\[
        \|s\|^2_{q u}:=|s|_\omega^2\e^{-q u}.
\]
\end{proposition}

Lastly, we will use the following construction (cf. \cite[Lemma 8]{DPS}) of cutoff functions to integrate-by-parts.

\begin{lemma}\label{lem:cut-off}
Assume $(X,\omega)$ satisfies either curvature hypothesis in Theorem \ref{thm:main}. Fix $o\in X$. Then there exist $0< \theta < 1$, $A>0$ and $a_0>0$ such that the following holds: for all $a>a_0$ there exist a smooth function $\chi_a:X\rightarrow [0,1]$ having the following properties:
\begin{enumerate}
\item $\chi_a \equiv 1$ on $B(o,\theta a)$ and $\mathrm{Supp}(\chi_a)\subset B(o,\theta^{-1}a).$
\item There exists a constant $A$ such that $$|\nabla \chi_a|, |\ddc \chi_a| \leq \frac{A}{a}. $$
\end{enumerate}
\end{lemma}


\section{Proof of Theorem \ref{thm:main}} \label{sec:proof}

If $n=1$, the assertion is Cohn--Vossen's theorem, since the Ricci form is a constant multiple of the Gaussian curvature form.  We therefore assume $n\geq2$.
By Proposition \ref{prop:weight}, there exists a uniformly Lipschitz strictly plurisubharmonic function $\varphi$ with
\[
       0< \int_X\alpha^n<\infty,
        \qquad \alpha=\ddbar\varphi .
\]
Choose $q$ and $s\in H^0(X,K_X)\setminus\{0\}$ as in Proposition \ref{prop:section}.

For $\vep>0$ set
\[
        f_\vep:=\frac{\Vert s\Vert_{q\varphi}^2}{\vep^2},\qquad
        V_\vep:=\log(1+f_\vep),\qquad
        \zeta_\vep(s):=\ddbar V_\vep+q\alpha .
\]
For $0<t\leq1$ we also set
\[
        f_{\vep,t}:=\frac{\Vert s\Vert_{q u_t}^2}{\vep^2},\qquad
        V_{\vep,t}:=\log(1+f_{\vep,t}),\qquad
        \alpha_t:=\ddbar u_t,
\]
\[
        \zeta_{\vep,t}(s):=\ddbar V_{\vep,t}+q\alpha_t .
\]
On compact subsets, $u_t\to\varphi$, $f_{\vep,t}\to f_\vep$, and
$V_{\vep,t}+qu_t\to V_\vep+q\varphi$ locally uniformly.  Hence, for every compactly supported smooth function $\eta$, Bedford--Taylor continuity gives
\begin{equation}\label{eq:BT-convention}
\int_X \eta\, f_{\vep,t}\zeta_{\vep,t}(s)^k\alpha_t^m\omega^l
\longrightarrow
\int_X \eta\, f_\vep\zeta_\vep(s)^k\alpha^m\omega^l
\end{equation}
whenever $k+m+l=n$; the same convergence holds without the factor $f_{\vep,t}$.  We use this regularisation below for integrations by parts involving $V_\vep$.

\begin{lemma}\label{lem:PL}
$\zeta_\vep(s)$ is a closed, positive $(1,1)$ current satisfying
\begin{equation}\label{ineq:PLinequality}
\zeta_\vep(s)\geq
        \frac{\Vert s\Vert_{q\varphi}^2}{\Vert s\Vert_{q\varphi}^2+\vep^2}\,\Ric_\omega .
\end{equation}
Moreover,
\begin{equation}\label{ineq:RicciPL}
        \int_X\Ric_\omega^n
        \leq
        \liminf_{\vep\to0^+}\int_X\zeta_\vep(s)^n,
\end{equation}
where the wedge product on the right is interpreted in the Bedford--Taylor sense.
\end{lemma}

\begin{proof}
For $t>0$, the forms are smooth.  Away from the zero set of $s$,
\[
        \ddbar\log f_{\vep,t}=\Ric_\omega-q\alpha_t,
\]
so
\[
\zeta_{\vep,t}(s)
 =\frac{f_{\vep,t}}{1+f_{\vep,t}}\Ric_\omega
 +\frac{q}{1+f_{\vep,t}}\alpha_t
 +\frac{\sqrt{-1}\,\partial f_{\vep,t}\wedge\bar\partial f_{\vep,t}}
        {f_{\vep,t}(1+f_{\vep,t})^2}
 \geq
 \frac{f_{\vep,t}}{1+f_{\vep,t}}\Ric_\omega .
\]
The quotient term extends smoothly across $Z(s)$: locally $f_{\vep,t}=\rho |g|^2$ with $\rho>0$ smooth and $g$ holomorphic.  Passing $t\to0$ gives \eqref{ineq:PLinequality}.

Let $\eta\ge0$ be compactly supported and smooth.  By pointwise algebraic monotonicity for smooth semipositive forms,
\[
\int_X\eta\,\zeta_{\vep,t}(s)^n
\geq
\int_X\eta\left(\frac{f_{\vep,t}}{1+f_{\vep,t}}\right)^n\Ric_\omega^n .
\]
Letting $t\to0$ and using \eqref{eq:BT-convention} gives
\[
\int_X\eta\,\zeta_\vep(s)^n
\geq
\int_X\eta\left(\frac{f_\vep}{1+f_\vep}\right)^n\Ric_\omega^n .
\]
As $\vep\downarrow0$, the coefficient on the right increases to $1$ off the divisor $Z(s)$, which has zero Riemannian measure.  Fatou's lemma therefore gives
\[
\int_X\eta\,\Ric_\omega^n
\leq
\liminf_{\vep\to0^+}\int_X\eta\,\zeta_\vep(s)^n
\leq
\liminf_{\vep\to0^+}\int_X\zeta_\vep(s)^n .
\]
Taking $\eta=\chi_a^{4n}$ and letting $a\to\infty$ proves \eqref{ineq:RicciPL}.
\end{proof}

We next estimate the right-hand side of \eqref{ineq:RicciPL}.  By Fatou's lemma and the cutoffs,
\begin{align}
\int_X\zeta_\vep(s)^n
&\leq \liminf_{a\to\infty}\int_X\chi_a^{4n}\zeta_\vep(s)^n                                      \notag\\
&\leq \limsup_{a\to\infty}\int_X
        \left(V_\vep\ddbar(\chi_a^{4n})+q\alpha\chi_a^{4n}\right)
        \zeta_\vep(s)^{n-1}                                                   \notag\\
&\leq \limsup_{a\to\infty}\int_X
        \left(\frac{C}{a}V_\vep\omega\chi_a^{4n-2}+q\alpha\chi_a^{4n}\right)
        \zeta_\vep(s)^{n-1}.\label{eq:firstexpansionofzeta}
\end{align}
Here the integration by parts is first done for the smooth $t>0$ regularizations and then $t\to0$ is taken; also
$|\ddbar(\chi_a^N)|\leq Ca^{-1}\chi_a^{N-2}\omega$ for the exponents used below.

One can now hope that inductively the only surviving term is the Monge-Amp\`ere mass of $\alpha$ and that every other term goes to $0$ for fixed $\vep$. The problem with such an expectation is seen right away for the term $$\displaystyle \int_X \chi_a^{4n-2} V_{\vep} \omega \zeta_\vep(s)^{n-1}=\int_X V_{\vep} \chi_a^{4n-2} (\ddbar V_{\vep}+q\alpha)  \omega \zeta_\vep(s)^{n-2}.$$
Pretending $V_{\vep}$ is smooth, if we integrate-by-parts twice, it is not helpful because we end up with two derivatives of $V_{\vep}$. If we transfer only one derivative, we end up with a term involving $V_{\vep} \bar{\partial} V_{\vep}$. In \cite{DPS}, using the Cauchy-Schwarz inequality and a delicate $W^{1,2}$-type estimate on $s$, we could bound such a term. However, when $n>2$, another iteration of this argument will run into difficulties involving higher-order derivatives of $s$. Here, motivated by Ohsawa-Takegoshi type arguments, ChatGPT 5.5 Pro produced a function $H$ satisfying good properties that allow an induction argument. Indeed, define
\begin{equation}\label{def:Hep}
        H(r):=r-\log(1+r),\qquad r\ge0 .
\end{equation}
For every smooth function $h\ge0$,
\begin{equation}\label{ineq:propsofHep}
        0\leq H(h)\leq h,
        \qquad
        h\,\ddbar\log(1+h)\leq\ddbar H(h),
\end{equation}
indeed the difference in the second inequality is
$\sqrt{-1}\,\partial h\wedge\bar\partial h/(1+h)$.

\begin{lemma}\label{lem:mixedterms}
Fix $\vep>0$.  If $1\leq l\leq n$, $0\leq k,m\leq n-1$, and $k+l+m=n$, then there is a constant $C_{k,l,m}(\vep)$, independent of $a$, such that
\begin{equation}\label{ineq:mixedtermsineqone}
\int_X\chi_a^{4n-2l}f_\vep\zeta_\vep(s)^k\alpha^m\omega^l
\leq C_{k,l,m}(\vep).
\end{equation}
Moreover, for $0\leq p\leq n-1$ there is a constant $D_p(\vep)$, independent of $a$, such that
\begin{equation}\label{ineq:mixedtermsineqtwo}
\int_X\chi_a^{4n}\zeta_\vep(s)^p\alpha^{n-p}
\leq \frac{D_p(\vep)}{a}+q^p\int_X\alpha^n .
\end{equation}
\end{lemma}

Assuming Lemma \ref{lem:mixedterms}, \eqref{eq:firstexpansionofzeta} gives
\begin{align}
\int_X\zeta_\vep(s)^n
&\leq
\limsup_{a\to\infty}\left\{
\frac{C}{a}\int_X\chi_a^{4n-2}f_\vep\omega\zeta_\vep(s)^{n-1}
+q\int_X\chi_a^{4n}\alpha\zeta_\vep(s)^{n-1}\right\}       \notag\\
&\leq q^n\int_X\alpha^n .\label{ineq:zeta-uniform}
\end{align}
The first term vanishes by \eqref{ineq:mixedtermsineqone} with $(k,m,l)=(n-1,0,1)$, and the second term uses \eqref{ineq:mixedtermsineqtwo} with $p=n-1$.  Combining \eqref{ineq:zeta-uniform} with Lemma \ref{lem:PL} proves Theorem \ref{thm:main}.

\begin{proof}[Proof of Lemma \ref{lem:mixedterms}]
Assume first that \eqref{ineq:mixedtermsineqone} has been proved.  We prove \eqref{ineq:mixedtermsineqtwo} by induction on $p$.  The case $p=0$ is immediate, with $D_0(\vep)=0$.  Suppose the estimate is known at level $p$.  Using regularisation and $V_\vep\leq f_\vep$,
\begin{align*}
\int_X\chi_a^{4n}\zeta_\vep(s)^{p+1}\alpha^{n-p-1}
&=\int_X\chi_a^{4n}\zeta_\vep(s)^p(\ddbar V_\vep+q\alpha)\alpha^{n-p-1}\\
&\leq \frac{C}{a}\int_X\chi_a^{4n-2}f_\vep\zeta_\vep(s)^p\alpha^{n-p-1}\omega
+q\int_X\chi_a^{4n}\zeta_\vep(s)^p\alpha^{n-p}\\
&\leq \frac{D_{p+1}(\vep)}{a}+q^{p+1}\int_X\alpha^n,
\end{align*}
where the first term is absorbed into $D_{p+1}(\vep)/a$ by \eqref{ineq:mixedtermsineqone}.  This proves \eqref{ineq:mixedtermsineqtwo}.

It remains to prove \eqref{ineq:mixedtermsineqone}.  We prove the following regularized estimate, uniform in the smoothing parameter: for $0<t\le1$ and $k+m+l=n$,
\begin{equation}\label{ineq:mixedtermsineqone-reg}
\int_X\chi_a^{4n-2l}f_{\vep,t}\zeta_{\vep,t}(s)^k\alpha_t^m\omega^l
\leq C_{k,l,m}(\vep),
\end{equation}
with $C_{k,l,m}(\vep)$ independent of $a$ and $t$.  Then \eqref{ineq:mixedtermsineqone} follows from \eqref{eq:BT-convention}.

We induct on $j=k+m=n-l$.  When $j=0$,
\[
\int_X\chi_a^{2n}f_{\vep,t}\omega^n
\leq \frac1{\vep^2}\int_X\Vert s\Vert_{q u_t}^2\omega^n
\leq \frac{C}{\vep^2}
\]
by Proposition \ref{prop:section}, uniformly in $a$ and $t$.
Assume \eqref{ineq:mixedtermsineqone-reg} is known for all triples with $k+m\leq j$.  Let $(k,m,l)$ satisfy $k+m=j$ and $k+m+l=n$.  If $l=1$ there is no level $j+1$ estimate to prove, so assume $l\ge2$.

First consider the triple $(k,m+1,l-1)$.  Put
\[
        N:=4n-2l+2,
        \qquad
        T_t:=\zeta_{\vep,t}(s)^k\alpha_t^m\omega^{l-1},
\]
and
\[
        J_{a,t}:=\int_X\chi_a^N f_{\vep,t}\alpha_t\wedge T_t .
\]
Since $\alpha_t=\ddbar u_t$ and $T_t$ is closed, integration by parts and absolute values give
\begin{align}
J_{a,t}
&\leq \frac{C}{a}\int_X\chi_a^{N-2}f_{\vep,t}\zeta_{\vep,t}(s)^k\alpha_t^m\omega^l
+
\left|\int_X\chi_a^N\sqrt{-1}\,\partial f_{\vep,t}\wedge\bar\partial u_t\wedge T_t\right| .
\label{ineq:regularized-J-first}
\end{align}
The first term is bounded by the induction hypothesis.  For the second term, Cauchy--Schwarz gives
\begin{align}
&\left|\int_X\chi_a^N\sqrt{-1}\,\partial f_{\vep,t}\wedge\bar\partial u_t\wedge T_t\right|        \notag\\
&\leq
\left(\int_X\chi_a^N
\frac{\sqrt{-1}\,\partial f_{\vep,t}\wedge\bar\partial f_{\vep,t}}{f_{\vep,t}}
\wedge T_t\right)^{1/2}
\left(\int_X\chi_a^N f_{\vep,t}\sqrt{-1}\,\partial u_t\wedge\bar\partial u_t\wedge T_t\right)^{1/2}.
\label{ineq:regularized-CS}
\end{align}
The quotient is interpreted by its smooth extension across $Z(s)$; locally $f_{\vep,t}=\rho |g|^2$ with $\rho>0$ smooth and $g$ holomorphic.  Since $|\nabla u_t|\leq C$,
$\sqrt{-1}\,\partial u_t\wedge\bar\partial u_t\leq C\omega$, and since $0\leq\chi_a\leq1$,
$\chi_a^N\leq\chi_a^{N-2}=\chi_a^{4n-2l}$.  Hence the second factor in \eqref{ineq:regularized-CS} is bounded by the induction hypothesis, and
\begin{equation}\label{ineq:regularized-CS-reduced}
\left|\int_X\chi_a^N\sqrt{-1}\,\partial f_{\vep,t}\wedge\bar\partial u_t\wedge T_t\right|
\leq
C\left(\int_X\chi_a^N
\frac{\sqrt{-1}\,\partial f_{\vep,t}\wedge\bar\partial f_{\vep,t}}{f_{\vep,t}}
\wedge T_t\right)^{1/2}.
\end{equation}
The Bochner identity for the canonical section $s$ is
\[
\ddbar f_{\vep,t}
=f_{\vep,t}\Ric_\omega-qf_{\vep,t}\alpha_t
+\frac{\sqrt{-1}\,\partial f_{\vep,t}\wedge\bar\partial f_{\vep,t}}{f_{\vep,t}}.
\]
Since $\Ric_\omega\geq0$,
\begin{align}
\int_X\chi_a^N
\frac{\sqrt{-1}\,\partial f_{\vep,t}\wedge\bar\partial f_{\vep,t}}{f_{\vep,t}}\wedge T_t
&\leq
\int_X\chi_a^N\ddbar f_{\vep,t}\wedge T_t+qJ_{a,t}                                      \notag\\
&=\int_X f_{\vep,t}\ddbar(\chi_a^N)\wedge T_t+qJ_{a,t}                                  \notag\\
&\leq C+qJ_{a,t},\label{ineq:regularized-bochner-estimate}
\end{align}
where the last line uses the cutoff estimate and the induction hypothesis.  Combining \eqref{ineq:regularized-J-first}, \eqref{ineq:regularized-CS-reduced}, and \eqref{ineq:regularized-bochner-estimate},
\[
        J_{a,t}\leq C+C(C+qJ_{a,t})^{1/2}\leq C+C\sqrt{J_{a,t}},
\]
so $J_{a,t}\leq C$, uniformly in $a$ and $t$.

Now consider the triple $(k+1,m,l-1)$.  Using \eqref{ineq:propsofHep} at the smooth level,
\begin{align}
&\int_X\chi_a^N f_{\vep,t}\zeta_{\vep,t}(s)^{k+1}\alpha_t^m\omega^{l-1}      \notag\\
&\leq
\int_X\chi_a^N\zeta_{\vep,t}(s)^k\ddbar H(f_{\vep,t})\alpha_t^m\omega^{l-1}
+q\int_X\chi_a^N f_{\vep,t}\zeta_{\vep,t}(s)^k\alpha_t^{m+1}\omega^{l-1}      \notag\\
&\leq\left|\int_X H(f_{\vep,t})\ddbar(\chi_a^N)\wedge\zeta_{\vep,t}(s)^k\alpha_t^m\omega^{l-1}\right|
+q\int_X\chi_a^N f_{\vep,t}\zeta_{\vep,t}(s)^k\alpha_t^{m+1}\omega^{l-1}      \notag\\
&\leq
\frac{C}{a}\int_X\chi_a^{N-2}f_{\vep,t}\zeta_{\vep,t}(s)^k\alpha_t^m\omega^l
+q\int_X\chi_a^N f_{\vep,t}\zeta_{\vep,t}(s)^k\alpha_t^{m+1}\omega^{l-1}
\leq C.\label{ineq:regularized-kplusone}
\end{align}
The last line uses the induction hypothesis for the first term and the already proved $(k,m+1,l-1)$ estimate for the second.  This completes the induction, proves \eqref{ineq:mixedtermsineqone-reg}, and hence proves \eqref{ineq:mixedtermsineqone}.
\end{proof}

\end{document}